\documentclass{article} 
\usepackage{amsfonts,amsmath,amssymb,textcomp,enumerate,bbm}
\usepackage{amsthm}
\newcommand{\R}{\mathbb{R}}
\DeclareMathOperator{\Ric} {Ric}
\DeclareMathOperator{\cut} {cut}
\newtheorem{theorem}{Theorem}
\newtheorem{proposition}{Proposition}
\newtheorem{lemma}{Lemma}
\newtheorem{corollary}{Corollary}
\newtheorem{remark}{Remark}
\newcommand{\deriv}{d}
\newcommand{\supp}{\mathop{\mathrm{supp}}}
\usepackage[english]{babel}
\newcommand{\assign}{:=}
\newcommand{\nequiv}{\not\equiv}

\newcommand{\tmop}[1]{\ensuremath{\operatorname{#1}}}
\newenvironment{enumeratealpha}{\begin{enumerate}[a{\textup{)}}] }{\end{enumerate}}
\title{Sharp dimension-free weighted bounds for the Bakry-Riesz vector} 
\author{Kamilia DAHMANI} 
\begin{document} 
\maketitle

\textbf{Abstract.} We prove a sharp dimensionless weighted estimate of the Riesz vector on a Riemannian manifold with non-negative Bakry-Emery curvature. The proof is by the method of Bellman functions, where the explicit expression of a Bellman function of six variables is essential.\\
Notice that our estimate is terms of the Poisson characteristic of the weight includes the case of the Gauss space as well as other spaces that are not necessarily of homogeneous type.

\section{Introduction}
The classical Hilbert transform can be defined via the equation $H = \partial
\circ (- \Delta)^{- 1 / 2}$ and the Riesz vector via $\mathcal{R} = d \circ (-\Delta  )^{-1/2}$. Given a complete Riemannian manifold $(X,g,\mu)$
of non-negative Ricci curvature $\tmop{Ric}$, this definition extends via the
use of the Laplace Beltrami operator. If in addition, we endow the space with
the measure $e^{- \varphi} d \mu$ and the Bakry-Emery curvature
$\tmop{Ric}_{\varphi} = \tmop{Ric} + \nabla^2 \varphi$ is non-negative, then
the Riesz vector is defined as $\mathcal{R}_{\varphi} = d \circ (-\Delta _{\varphi} )^{-1/2}$
with $\Delta _{\varphi} = \Delta - \nabla \varphi \cdot \nabla$.

The focus in this note is on the Riesz vector $\mathcal{R}_{\varphi}$ on Riemannian
manifolds $(X,g,\mu)$ defined respecting the measures of the type
$e^{- \varphi} d \mu$ where an additional weight is present: in weighted
spaces $L^2 (\omega) = L^2 (\omega e^{- \varphi} d \mu)$ we study the operator norm
of the Riesz vector.

On a complete Riemannian manifold $(X,g,\mu)$ endowed with
measure $e^{- \varphi} d \mu$ and non-negative Bakry-Emery curvature, we
show dimension-free sharp weighted norm estimates for the arising Riesz vector
in terms of the Poisson flow $A_2$ characteristic of the weight. Such
estimates are known to be in sharp dependence on the power
of the Poisson characteristic even when $X =\mathbbm{D}$ and $\varphi = 0$. 

\

In 1973, Hunt-Muckenhoupt-Wheeden \cite{HMW} proved that the Hilbert
transform is bounded on $L^2 (\omega)$ if and only if the weight $\omega$ satisfies the
so-called $A_2$ condition:
\[ Q_2 (\omega) \assign \sup_I \frac{1}{| I |} \int_I \omega \frac{1}{| I |} \int_I \omega^{-
   1} < \infty \]
where the supremum runs over intervals $I$. It has then been extended to all
Calderon-Zygmund operators \cite{H} and beyond \cite{AM}. The classes of
weights considered are only defined in terms of volume of balls, so this
entire theory has been extended to the doubling framework. Notice that our
measures here in this note are not necessarily doubling.

We are interested in a version of the $A_2$ class with charateristic $Q_2 (\omega)$
which is particularly well-suited for working with the Riesz transforms on
$X$. Namely, we use the Poisson-$A_2$ class with characteristic $\widetilde{Q}_2
(\omega)$, which considers Poisson averages instead of box averages in the
definition of $A_2$. This allows us to tackle non-homogeneous measures as well
as obtain a sharp bound free of dimension for the Riesz vector:
\[ \| \mathcal{R}_{\varphi} \|_{L^2 (\omega) \rightarrow L^2 (\omega)} \lesssim \widetilde{Q}_2 (\omega) .
\]
We stress both the continuity of the operator in this setting, its rate of
contiunuity, i.e. the first power of the characteristic $\widetilde{Q}_2 (\omega)$ as
well as the fact that implied constants do not depend upon the dimension. The linear estimate (in terms of the classical
characteristic that induces a dimensional growth) is very recent in the case
of general $X$ with bounded geometry and $\varphi = 0$ \cite{BFP}. This
present result has therefore the following novelties:
\begin{enumeratealpha}
  \item A weighted estimate holds even in the case $\varphi \nequiv 0$ (and in
  the presence of non-homogeneity).
  
  \item The estimate is sharp in terms of dependence on the power of
  $\widetilde{Q}_2
(\omega)$.
  
  \item The estimate is free of dimension.
\end{enumeratealpha}
In the last fifteen years, it has been of great interest to obtain optimal
operator norm estimates in Lebesgue spaces endowed with Muckenhoupt weights in
the sense of b). One asks for the growth of the norm of certain operators,
such as the Hilbert transform with respect to a characteristic assigned to the
weight, usually using the classical characteristic. Originally, the main
motivation for sharp estimates of this type came from certain important
applications to partial differential equations. See for example Fefferman-Kenig-Pipher \cite{FKP} and Astala-Iwaniec-Saksman \cite{AIS}.
Indeed, a long standing regularity problem has been solved through the optimal
weighted norm estimate of the Beurling-Ahlfors operator, a classical
Calderon-Zygmund operator, using the heat flow characteristic of the weight. See Petermichl-Volberg \cite{PV}. Since then, the area has been developing
rapidly.

Questions of such optimal norm estimates have become known as $A_2$
conjectures. This conjecture was solved by Petermichl-Volberg \cite{PV} in
the case of the Beurling-Ahlfors operator, by Petermichl in \cite{Petermichl}
for the Hilbert transform and then by Hytonen in \cite{H} for arbitrary
Calderon-Zygmund operators. Only very recently the sharp weighted theory has
been extended beyond the Calderon-Zygmund theory to so-called non-kernel
operators \cite{BFP} that contain Riesz transforms on manifolds (although
without the measure $e^{- \varphi}$).

The search for optimal estimates in weighted spaces has greatly improved our
understanding of central operators such as the Riesz transforms and has
developped numerous tools with a probabilistic flavor. Notably, the first
solution of an $A_2$ problem \cite{PV} uses an underlying estimate for
predictable martingale multipliers of dyadic martingales under a change of law
by Wittwer \cite{witt}. Her proof in turn uses important developments on a
corresponding two weight question in \cite{NTV}, using the Bellman function
method. It has been known at least since Gundy-Varopoulos \cite{GV} that the
Riesz transforms of a function can be written as conditional expectation of a
simple transformation of a martingale associated to the function. This is a
simple but beautiful fact that follows from a Littlewood-Paley type formula
for the Riesz vector using the Poisson flow. This is the reason behind the
availability of a certain transference method using said Bellman functions, a
control strategy for martingales. A beautiful observation by Bakry
\cite{bak} facilitates this transference in the given setting to manifolds.

This transference approach has also been used for the Hilbert transform on the
disk \cite{PW} in the early days of the sharp weighted theory and for the
Riesz vector in Euclidean space \cite{PDW} to obtain a sharp dimensionless
estimate with respect to the well adapted Poisson $A_2$ characteristic.

The Poisson $A_2$ characteristic arises naturally from the viewpoint of
martingales driven by space-time Brownian motion as in the representation
formulae for Riesz transforms. The Poisson $A_2$ class is adapted to the
stochastic process considered. Random walks and Poisson flows on manifolds are
a delicate matter and we refer the reader to the excellent text by Emery
\cite{Emery}.

Even in the case $\varphi = 0$, these operators are not necessarily of
Calderon-Zygmund type. These Riesz transforms fit into the class of non-kernel
operators, whose weighted theory was established in Auscher-Martell
\cite{AM}. The optimal weigted norm estimates for these types of operators, even without the extra $e^{- \varphi}$, have only recently been found in
\cite{BFP} (in terms of the classical charateristic). In all these proofs,
the doubling feature of the measure $\mu= e^{- 0} {\textmu}$ is heavily
used and dimensional growth occurs. Measures of the type considered in this
note can be non-doubling, such as in the case of the Gauss space, which is a
model example for our setting.

There is a spiked recent interest in weighted estimates in the non-doubling
setting. Very recently, the scalar and vector valued predictable discrete in
time martingale multipliers in the general non-homogeneous setting have been
considered, using the corresponding martingale $A_2$ characteristic. There are
two independent proofs by Thiele-Treil-Volberg \cite{TTV} as well as Lacey
\cite{L} for discrete in time filtrations on general probability measures.
See \cite{PD} and \cite{DP2} for continuous in time martingales. There have
been observations of failure of estimates for certain crucial operators,
called Haar Shifts, one of the earlier tools in the sharp weighted theory
\cite{LMP}. Very recently in \cite{V+} an
estimate for Calderon-Zygmund operators in terms of $A_2$ characteristic on
cells has been found using so-called sparse operators. The question remains
open, even for Calderon-Zygmund operators in non-doubling Euclidean space, in
terms of the classical characteristic using balls. The result here uses the
more forgiving bumped Poisson $A_2$ characteristic but is in turn free of
dimension and remains sharp in terms of the dependence on the weight's
charateristic. As mentioned before, the proof is by transference using Bellman
functions, resembling the strategy used by Carbonaro-Dragicevic \cite{DC} for
the unweighted case in $L^p$. Their proof relies on Bakry's idea of adapted
Poisson flow on one-forms \cite{bak} in combination with the concise but
powerful Bellman function of Nazarov-Treil \cite{NT} that is adapted to $L^p$
estimates. A key difference here is the complication of the weighted Bellman
function, that has to be known in a reasonably explicit manner. This Bellman
function of six variables is derived through an analysis of \cite{NTV} as
well as \cite{PW}. A similar function was constructed in \cite{PD}.
Properties that go beyond those needed for martingale multipliers are required
to obtain the desired Riesz transform estimates on manifolds, which is in a
sharp contrast to the Euclidean case \cite{PDW}, where existence of this
function suffices.\\
%

\section{Developement}

\subsection{Setting and notations}

	Let $(X,g,\mu)$ be a complete Riemannian manifold of dimension $N$ without boundary and such that constants are integrable. Let $\Delta$ be the
negative Laplace-Beltrami operator on $X$. Given $\varphi\in C^{2}(X)$,
consider the weighted (non-doubling) measure on $X$ defined by 
\begin{equation*}
d \mu_{\varphi}(x)=e^{-\varphi(x)}d \mu(x),
\end{equation*}
and denote by $\Delta_{\varphi}$ the associated weighted Laplacian defined on $%
C^{\infty}_{c}(X)$ by 
\begin{equation*}  
\Delta_{\varphi}f=\Delta f-\nabla f.\nabla \varphi.
\end{equation*}
Notice that for all $f,g \in C^{\infty} _c (X)$, we have 
\begin{equation} \label{page4}
\int_X (\nabla f, \nabla g)d \mu_{\varphi}(x)=-\int_X f\Delta_{\varphi}gd \mu_{\varphi}(x) =-\int_X g\Delta_{\varphi}fd \mu_{\varphi}(x).
\end{equation}
It was proved in \cite{bak} and \cite{Strich} that on complete Riemannian manifolds, the operator $\Delta_{\varphi}$  is essentially self-adjoint on $L^2(X,\mu _{\varphi})$. We will still note $\Delta_{\varphi}$ its unique self-adjoint extension. \\
The Bakry-Emery curvature tensor associated with $\Delta_{\varphi}$ is defined by
\[\Ric _{\varphi} = \Ric + \nabla ^2 \varphi,\]
where $\Ric$ denotes the Ricci curvature tensor on $X$ and $\nabla ^2 \varphi$ is a $2-$tensor. All over this paper, we will consider that $\Ric _{\varphi}\geq 0$.\\

Before we proceed, we recall some notions in differential geometry that will be useful.

For each $x$ in $X$, we denote the tangent space and its dual, the cotangent space at $x$ respectively by $T_xX$ and $T^*_xX$ so that 
\[TX=\cup _{x\in X} T_xX \text{ and } T^*X=\cup _{x\in X} T^*_xX.\]
We denote by $<\cdot,\cdot>$ either the inner product in $TM$ and $T^*X$, or in the Lebesgue space $L^2(X,\mu_{\varphi})$ with a subscript to avoid ambiguities.\\
By $\|\cdot\|_{L^{2}(X,\omega\mu_{\varphi})}$ and $\|\cdot%
\|_{L^{2}(T^{*}X,\omega^{-1}\mu_{\varphi})}$, we denote respectively the norm in $L^{2}(X,\omega\mu_{\varphi})$ and $L^{2}(T^{*}X,\omega^{-1}\mu_{\varphi})$, where $\omega$ and $\omega ^{-1}$ are weights that belong to $L^1_{loc}(X,\mu_{\varphi})$.

We denote by $d$ and $\nabla$ respectively the exterior and the covariant derivative and $d^*_{\varphi}$ and $\nabla ^*_{\varphi}$ their $L^2(\mu _{\varphi})$-adjoint operators. We also define $\overline{\nabla}$ as the total covariant derivative on $X \times \R_+$ that satisfies $|\overline{\nabla}\eta|=\sqrt{|\nabla\eta|^2+|\partial _t \eta|^2}$, for all $\eta$ in $C^{\infty}(T^*X(X\times \R_+))$.

We consider $\vec{\Delta}_{\varphi}=-(dd^*_{\varphi}+d^*_{\varphi}d)$ to be the weighted Hodge-De Rham Laplacian acting on $1-$forms. As for the Laplace Beltrami operator, $\vec{\Delta}_{\varphi}$ initially defined on smooth $1-$forms with compact support is essentially self-adjoint on $L^{2}(T^{*}X,\mu_{\varphi})$ and again, we will denote $\vec{\Delta}_{\varphi}$ its unique self-adjoint extension.

Finally, we set 
\[P_t=\exp (-t(-\Delta_{\varphi})^{1/2}) \text{ and } \vec{P}_t=\exp (-t(-\vec{\Delta}_{\varphi})^{1/2}).\]
Note that the semigroup $P_t$ acting on functions is an integral operator with positive kernel that we will note $p_t$ \cite{Strich}.\\

We are concerned in this paper with a special class of weights, called Poisson$-A_2$ and noted $\widetilde{A}_2$. We say that $\omega \in \widetilde{A}_2$ if and only if
\[\widetilde{Q}_2(\omega):=\sup _{(x,t) \in X \times \R _+}P_t(\omega)(x)P_t(\omega ^{-1})(x)< \infty. \]
The weights involved in this definition are a priori in $L^2(X,\mu_{\varphi})$.

\begin{remark} \label{remark1}
For $\omega \in L^1_{loc}(X,\mu_{\varphi})$, we define its two-sided truncation 
\[\omega _n = n^{-1} \chi _{\omega \leq n^{-1}} + \omega \chi _{n^{-1} \leq \omega \leq n} + n \chi _{\omega \geq n}, \]
where $\chi$ is the characteristic function and $n \in \mathbb{N}^*$. The truncated weight $\omega _n$ is clearly in $L^2(X,\mu_{\varphi})$ and satisfies some interesting properties that we are going to see later. For the moment, we are going to work with $\omega _n$ and then extend our results to $\omega$, including a definition for $\widetilde{Q}_2(\omega)$ when $\omega$ is only locally integrable. We are also going to suppose that $P_t\omega_n$ and $P_t\omega_n ^{-1}$ are finite almost everywhere so that $\widetilde{Q}_2(\omega_n)$ makes sense. 
\end{remark}

\begin{remark}
Throughout this paper, $C$ will denote constants whose values may change even in a chain of inequalities. These constants are independant of the dimension of the manifold and other important quantities.
\end{remark}

\subsection{Preliminaries}
The following lemma slightly differs from the one appearing in \cite{bak} since it involves weights. The stated results will be of great utility in Sections 3 and 5.
\begin{lemma}
\label{lemma1} For every $f\in C_{c}^{\infty }\left( X\right)$, $\vec{g} \in C_{c}^{\infty }\left( T^{\ast }X\right)$ and $\omega \in L^2\left( X, \mu_{\varphi}\right)$,

\begin{enumerate}

\item[a)] $\left\vert P_{t}f\left( x\right) \right\vert ^{2}\leq
P_{t}\left( \left\vert f\right\vert ^{2}\omega\right)\left( x\right) P_t \omega ^{-1}(x). $

\item[b)] $dP_{t}  f =\vec{P}_{t}d  f   $ \newline
If we also have $\Ric_{\varphi }\geq 0$ then

\item[c)] $\left\vert e^{t\vec{\Delta}_{\varphi }}\vec{g} \left( x\right)
\right\vert _{T^{\ast }_xX}\leq e^{t\Delta_{\varphi }}\left\vert \vec{g}
\left( x\right) \right\vert _{T_{x}^{\ast }X} $

\item[d)] $\left\vert \vec{P}_{t}\vec{g} \left( x\right) \right\vert
_{T^{\ast }_xX}^{2}\leq P_{t}\left(  \left\vert \vec{g} \right\vert ^{2}_{T_{x}^{\ast
}X}\omega ^{-1}_n\right)\left( x\right) P_t\omega_n(x). $

\end{enumerate}
\end{lemma}

\begin{proof}
Items (b) and (c) in the lemma have been proved in \cite{bak}.\\
For item (a), we use the integral expression of $P_t$ and H\"{o}lder's inequality.
\begin{eqnarray*}
P_tf(x)&=&\int_X p_t^{1/2}(x,y)f(y)\omega^{1/2}(y)p_t^{1/2}(x,y)\omega^{-1/2}(y)d\mu_{\varphi}(y)\\
&\leq & \left( \int_X p_t(x,y)|f(y)|^2\omega(y)d\mu_{\varphi}(y)\right)^{1/2}\times \left( \int_X p_t(x,y)\omega^{-1}(y)d\mu_{\varphi}(y)\right)^{1/2}.
\end{eqnarray*}
To conclude, simply raise to the power 2 the above inequality.\\
To prove item (d), note that by \cite[Inequality (1.4)]{bak}, one can write 
\[\left\vert \vec{P}_{t}\vec{g} \left( x\right) \right\vert_{T^{\ast }_xX} \leq P_t |\vec{g}|_{T^{\ast }_xX}(x).\]
The proof is then analogous to the one of item (a).
\end{proof}

\section{Bilinear embedding and its corollary}

In this section, we state the main result of this paper and its corollary
for the Riesz transform. The proof of Theorem \ref{thm1} will be given
in Section \ref{sec_main_result}.

\begin{theorem}
\label{thm1} Suppose that $X$ is a complete Riemannian manifold without boundary and such that constants are integrale. Suppose also that $\Ric_{\varphi}\geq0$ and that $\omega_n$ and $\omega _n^{-1}$ are a.e positive weights defined as in Remark \ref{remark1}. Then for all $f \in C_c^{\infty}(X)$ and $\vec{g} \in C_c^{\infty}(T^*X)$, we have the following dimension-free estimate
\begin{equation} \label{E1}
\int_{0}^{\infty} \int_{X}|\overline{\nabla} P_{t}f(x)||\overline{\nabla} 
\vec{P_{t}}\vec{g}(x)| t \, d\mu_{\varphi}(x) d t \leq
20\widetilde Q_{2}(\omega_n)\|f\|_{L^{2}(X,\omega_n\mu_{\varphi})} \|\vec{g}%
\|_{L^{2}(T^{*}X,\omega^{-1}_n\mu_{\varphi})}.
\end{equation}
\end{theorem}
Let $\mathcal{R}_{\varphi}$ denote the Riesz transform initially defined on the range space of $-\Delta_{\varphi }$, $R(-\Delta_{\varphi })$ by 
\[ \mathcal{R}_{\varphi} = d \circ (-\Delta _{\varphi} )^{-1/2}\]
and that extends to a contraction
\[\mathcal{R}_{\varphi}:\overline{R(-\Delta_{\varphi })} \rightarrow L^2(T^*X).\]
We have then the following corollary :

\begin{corollary} \label{coro}
\label{coro} Under the above conditions, 
\begin{equation*}
\|\mathcal{R}_{\varphi}f\|_{L^{2}(T^{*}X,\omega_n\mu_{\varphi})}\leq 80\widetilde
Q_{2}(\omega_n)\|f\|_{L^{2}(X,\omega\mu_{\varphi})},
\end{equation*}
for all $f\in \overline{L^{2}(X,\omega_n\mu_{\varphi})\cap R(-\Delta_{\varphi })}^{L^2}$.
\end{corollary}

\begin{proof}
As we mentioned in the introduction, the idea is to represent the Riesz transform by using Poisson semigroups on functions and differential forms. Indeed, for every $f\in L^{2}(X,\omega_n\mu_{\varphi})\cap R(-\Delta_{\varphi }) $ and $\vec{g} \in C_{c}^{\infty }\left( T^{\ast }X,\omega_n^{-1}\mu_{\varphi}\right) $ we have the well known fact%
\begin{equation}
\int_{X}\left\langle \mathcal{R}_{\varphi}f\left( x\right) ,\vec{g} \left( x\right)
\right\rangle d\mu _{\varphi }\left( x\right) =4\int_{0}^{\infty
}\int_{X}\left\langle dP_{t}f\left( x\right) ,\frac{d}{dt}\vec{P}%
_{t}\vec{g} \left( x\right) \right\rangle d\mu _{\varphi }\left( x\right)
tdt.  \label{3}
\end{equation}
Assuming the claim (\ref{3}), Corollary \ref{coro} follows immediately from Theorem \ref{thm1}
and the Cauchy-Schwarz inequality. To prove the claim (\ref{3}), consider
the function%
\[
h \left( t\right) =\left\langle \vec{P}_{t}\mathcal{R}_{\varphi}f,%
\vec{P}_{t}\vec{g} \right\rangle _{L^{2}\left( \mu _{\varphi
}\right) }. 
\]
Since $\int_{X}\left\langle \mathcal{R}_{\varphi}f,\vec{g} \right\rangle
_{L^{2}\left( \mu _{\varphi }\right) }=\varphi \left( 0\right) ,$ it
suffices to show that%
\begin{equation}
h \left( 0\right)=\int_0^{\infty}\varphi^{\prime\prime}(t) tdt =4\int_{0}^{\infty }\left\langle dP_{t}f,\frac{d%
}{dt}\vec{P}_{t}\vec{g} \right\rangle _{L^{2}\left( \mu _{\varphi
}\right) }tdt  \label{4}
\end{equation}\\
In order to prove the first equality in (\ref{4}), it is enough to show that both $\varphi
\left( t\right) $ and $t\varphi ^{\prime }\left( t\right) $ tend to zero as $%
t\rightarrow \infty $. First, note that by Lemma \ref{lemma1}, $\vec{P}_{t}\mathcal{R}_{\varphi}f=\mathcal{R}_{\varphi}P_{t}f$. Therefore, by the $L^{2}$ contractivity of both $\mathcal{R}_{\varphi}$ and $\vec{P}_{t}$,
\begin{equation}
\left\vert h \left( t\right) \right\vert \leq \left\Vert P_{t}f\right\Vert
_{L^{2}\left( \omega _n\mu _{\varphi }\right) }\left\Vert \vec{g} \right\Vert
_{L^{2}\left(\omega^{-1} _n\mu _{\varphi }\right) }. \label{h(t)}
\end{equation}
Since $f\in R\left( -\Delta_{\varphi }\right) ,$ the spectral theorem gives that $P_{t}f%
\rightarrow 0$ in $L^{2}\left( \omega_n \mu _{\varphi }\right) $ as $t\rightarrow
\infty $.\\
Similarly, Lemma \ref{lemma1} gives%
\begin{eqnarray*}
h ^{\prime }(t) &=& 2\langle (-\Delta_{\varphi })^{1/2}\vec{P}_{t}d(-\Delta_{\varphi})^{-1/2}f,\vec{P}_{t}\vec{g} \rangle _{L^{2}(\mu_{\varphi }) } \\
&=& 2\left\langle P_{t}f,P_{t}d^{\ast }_{\varphi}\vec{g} \right\rangle
_{L^{2}\left( \mu _{\varphi }\right) },
\end{eqnarray*}%
therefore $\lim_{t\rightarrow +\infty }t\left\vert \varphi ^{\prime }\left(
t\right) \right\vert =0$ as before.\\
The second equality in (\ref{4}) can be
verified by a straightforward calculation, again with the help of Lemma \ref{lemma1}. Indeed,
\[h^{\prime\prime}(t)=2\left(\langle\frac{d}{dt}P_tf,P_td^{*}_{\varphi}\vec{g}\rangle +\langle P_tf,\frac{d}{dt}P_td^{*}_{\varphi}\vec{g}\rangle\right).\]
By Lemma \ref{lemma1}, 
\[\langle\frac{d}{dt}P_tf,P_td^*_{\varphi}\vec{g}\rangle=\langle dP_tf,\frac{d}{dt}\vec{P}_t\vec{g}\rangle,\]
and
\begin{eqnarray*}
\langle P_tf,\frac{d}{dt}P_td^*_\varphi\vec{g}\rangle &=& \langle P_tf,d^*_{\varphi}\frac{d}{dt}\vec{P}_t\vec{g}\rangle \\
&=& \langle dP_tf,\frac{d}{dt}\vec{P}_t\vec{g}\rangle.
\end{eqnarray*}
Thus, we get the desired result.
\end{proof}

\section{Bellman function}

\label{sec_bell}

As mentioned before, the main tool used to prove Theorem \ref{thm1} is a
particular Bellman function that is constructed explicitely. A substantial part of its origin lies in the seminal paper by Nazarov, Treil and Volberg \cite{NTV}. It has been developped in \cite{witt}, \cite{PW} and \cite{PD}, with the first explicit expression in \cite{PD}. Our construction differs from the one in \cite{PD} in that this construction is slightly shorter and and gives better (and explicit) numerical constants.\\
In fact, for any $Q \geq 1$, we can exhibit a function $B_Q$ in domain 
\[D_Q=\{\mathcal{X}:=(Z,H,\zeta ,\eta ,r,s) : \zeta^2 \leq Zr, \langle \eta, \eta \rangle \leq Hs, 1 \leq rs \leq Q\},\]
which is a subset of $\R _+ \times \R _+ \times \R \times \R^N \times \R _+ \times \R _+$. The function $B_Q$ is globally in $C^1$ and piecewise in $C^2$ such that
\begin{equation} \label{size}
0\leq B_{Q} \leq 80(Z+H);  
\end{equation}%
\begin{equation} \label{deriv}
-d^{2}B_{Q} \geq \frac{4}{Q}|d\zeta | |d\eta|; \text{ where $B_Q$ is in $C^2$.}
\end{equation}%
Furthermore, $B_Q$ is radial in $\zeta$ and $\eta$ in the sense that 
\[B_Q(Z,H,\zeta ,\eta ,r,s)=\overline{B}_Q(Z,H,|\zeta| ,|\eta |,r,s.)\]
Consequently, the domain of $\overline{B}_Q$ is defined accordingly in $\mathbb{R}^6$. Writing $\nu = |\eta| \in \mathbb{R}_+$ we have in addition
\begin{equation} \label{sign}
\partial _{\nu }\overline{B}_{Q} \leq 0.
\end{equation}

\begin{remark}
We use a Bellman function involving real variables. As a consequence, Theorem \ref{thm1} and Corollary \ref{coro} hold for real-valued functions and forms. One may find the corresponding estimates for complex-valued functions and forms by separating the real and imaginary parts.
\end{remark}

\begin{remark}
The property (\ref{deriv}) means that for all 6-tuple $\mathcal{X}=(Z,H,\zeta ,\eta ,r,s)$ in $D_Q$, we have 
\begin{equation*}
\left\langle -d^{2}B_{Q}d \mathcal{X},d \mathcal{X} \right\rangle \geq \frac{4}{Q}|d \zeta||d \eta| .
\end{equation*}
\end{remark}

Consider the Bellman functions $B_{Q}=B_{1}+B_{2}+B_{3}+B_{4}$ and $\overline{B}_{Q}=\overline{B}_{1}+\overline{B}_{2}+\overline{B}_{3}+\overline{B}_{4}$ of six variables such that

\begin{enumerate}
\item[$\bullet $] $B_{1}(Z,H,\zeta ,y,r,s)=Z-\dfrac{\zeta ^{2}}{r}+H-\dfrac{%
\langle \eta, \eta \rangle}{s},$

\item[$\bullet $] $B_{2}(Z,H,\zeta ,\eta ,r,s)=Z-\dfrac{\zeta ^{2}}{r}+H-%
\dfrac{\langle \eta, \eta \rangle}{s+\dfrac{M(r,s)}{Q^2} },$

\item[$\bullet $] $B_{3}(Z,H,\zeta ,\eta ,r,s)=Z-\dfrac{\zeta ^{2}}{r+\dfrac{N(r,s)}{Q^2} }%
+H-\dfrac{\langle \eta, \eta \rangle}{s},$\newline
where 
\begin{equation*}
M(r,s)=-\frac{4Q^2}{r}-rs^2+(4Q^2+1)s
\end{equation*}%
and 
\begin{equation*}
N(r,s)=-\frac{4Q^2}{s}-sr^2+(4Q^2+1)r.
\end{equation*}

\item[$\bullet$] $B_{4}=B_{41}+B_{42}+B_{43} \text{ with }$

\begin{itemize}
\item $B_{41}(Z,H,\zeta ,\eta ,r,s)=Z-\dfrac{\zeta ^{2}}{r+\dfrac{\widetilde{M}(r,s)}{Q} }+H-\dfrac{\langle \eta ,\eta\rangle}{s}$

\item $B_{42}(Z,H,\zeta ,\eta ,r,s)=Z-\dfrac{\zeta ^2}{r} +H-\dfrac{\langle \eta, \eta \rangle}{s+\dfrac{\widetilde{N}(r,s)}{Q}}$

\item $B_{43}(Z,H,\zeta ,\eta ,r,s)= \sup_{a>0} \left( Z-\dfrac{\zeta ^2}{r+a\dfrac{K(r,s)}{Q}}+H-\dfrac{\langle \eta ,\eta \rangle}{s+a^{-1}\dfrac{K(r,s)}{Q} } \right)$\newline
where\newline
$K(r,s)=\sqrt{Q}\sqrt{rs}-\dfrac{rs}{4}$,\newline
$\widetilde{M}(r,s)=-\dfrac{4Q}{s}-\dfrac{r^2s}{4Q}+(4Q+1)r$\newline
and\newline
$\widetilde{N}(r,s)=-\dfrac{4Q}{r}-\dfrac{s^2r}{4Q}+(4Q+1)s.$ 
\end{itemize}
\end{enumerate}

Next, we present properties on size and derivative estimates of the functions $M,N,K,\widetilde{M}
$ and $\widetilde{N}$, using the relation $1\leq rs \leq Q$ :\newline
\textbf{Functions $M$ and $N$:} 
\begin{equation*}
0\leq M\leq 5Q^2s\text{ and }-d^{2}M\geq r(ds)^{2},
\end{equation*}%
\begin{equation*}
0\leq N\leq 5Q^2r\text{ and }-d^{2}N\geq s(dr)^{2}.
\end{equation*}%
\newline
\textbf{Function $K$:} 
\begin{equation*}
0\leq K\leq Q\text{ and }-d^{2}K\geq \dfrac{1}{4}|dsdr|.
\end{equation*}%
\newline
\textbf{Functions $\widetilde{M}$ and $\widetilde{N}$:} 
\begin{equation*}
0\leq \widetilde{M}\leq 5Qr\text{ and }-d^{2}\widetilde{M}\geq \frac{|dsdr|%
}{s},
\end{equation*}%
\begin{equation*}
0\leq \widetilde{N}\leq 5Qs\text{ and }-d^{2}\widetilde{N}\geq \frac{|dsdr|%
}{r}.
\end{equation*}%

Now, define $\Pi=\lbrace \frac{K}{Q} = \frac{\zeta s}{|\eta|} \rbrace \cup \lbrace \frac{K}{Q} =  \frac{|\eta| r}{\zeta}\rbrace $. The function $B_{Q}$ satisfies the following :
\begin{lemma} \label{lemma3}
For every $(Z,H,\zeta ,\eta ,r,s)$ in $D_Q$,
\begin{itemize}
\item[1)] $0 \leq B_{Q} \leq 6(Z+H)$;
\item[2)] $\partial _{\nu }\overline{B}_{Q,} \leq 0$;
\item[3)] If $(Z,H,\zeta ,\eta ,r,s) \in D_Q \setminus \Pi$, then we have 
$-d^{2}B_{Q} \geq \frac{4}{Q}|d\zeta | |d\eta|$.
\end{itemize}
\end{lemma}

\begin{proof}
\begin{enumerate}
\item The result follows directly due to the construction of $B_{Q}$ as well as the hypothesis on $D_Q$.

\item For more convenience, we will deal with each function $\overline{B}_i, i=1,\dots,4$ separately.\\
It is clear that the derivatives in the variable $\nu$ of $\overline{B}_1$, $\overline{B}_{2}$, $\overline{B}_{3}$, $\overline{B}_{41}$ and $\overline{B}_{42}$ are negative by straightforward computations. It remains to study $\overline{B}_{43}$. Let us rewrite it in the form
\[\overline{B}_{43}(Z,H,|\zeta| ,|\eta| ,r,s)=Z+H-\sup_{a>0} \beta (a,Z,H,|\zeta |,|\eta| ,r,s),\]
with
\[\beta(a,Z,H,|\zeta| ,|\eta| ,r,s)=\dfrac{\zeta ^2}{r+aK(r,s)/Q}+\dfrac{\nu^2}{s+a^{-1}K(r,s)/Q}.\]
The function $\beta$ is continuously differentiable in $a>0$ and
\[\dfrac{\partial \beta}{\partial a}=-\dfrac{\zeta ^2 K/Q}{(r+aK/Q)^2}+\dfrac{\nu^2 K/Q}{(as+K/Q)^2},\]
which yields to
\[\dfrac{\partial \beta}{\partial a}=0 \Leftrightarrow a=\dfrac{Qr\nu-K\zeta}{Qs\zeta-K\nu},\]
provided this fraction is finite and non-null.  \\
Let $a_m:=\frac{Qr\nu-K\zeta}{Qs\zeta-K\nu}$. If both numerator and denominator are positive, $\partial \beta/ \partial a$ changes sign from positive to negative. Which means that the extremum is a maximum and it is attained at $a_m$. In this case,
\[\overline{B}_{43}=Z-\dfrac{\zeta ^2}{r+\dfrac{K(r,s)}{Q}\dfrac{Qr\nu - \zeta K(r,s)}{Qs\zeta - \nu K(r,s)}}+H-\dfrac{ \nu ^2}{s+\dfrac{K(r,s)}{Q}\dfrac{Qs\zeta - \nu K(r,s)}{Qr\nu - \zeta K(r,s)}}.\]
If $a_m$ is respectively null or infinite, then $\overline{B}_{43}$ is respectively $Z+H -\dfrac{\zeta ^2}{r}$ and $Z+H-\dfrac{\nu^2}{s}$. \\
To compute $\partial_{\nu}\overline{B}_{43}$, consider a one-parameter family of functions
\[ \overline{B}_{43}^a(Z,H,|\zeta| ,|\eta| ,r,s)=Z+H-\beta(a,Z,H,|\zeta| ,|\eta| ,r,s).\]
When $a_m$ is strictly positive and finite, it is clear that $\overline{B}_{43}= \overline{B}_{43}^{a_m}$.  \\
By chain rule
\[\dfrac{\partial \overline{B}_{43}}{\partial {\nu}} =  \left.\dfrac{\partial \overline{B}_{43}^a}{\partial a}\right| _{a=a_m}\cdot\dfrac{\partial a}{\partial \nu} + \left.\dfrac{\partial \overline{B}_{43}^a}{\partial \nu}\right| _{a=a_m}.\]
But $\left.\dfrac{\partial B_{43}^a}{\partial a}\right| _{a=a_m}=0$, since $\beta$ attains its maximum at $a_m$. Consequently
\begin{eqnarray*}
\dfrac{\partial \overline{B}_{43}}{\partial {\nu}} &=&  \left.\dfrac{\partial \overline{B}_{43}^a}{\partial \nu}\right| _{a=a_m} \\
&=&  -\dfrac{2 \nu }{s+a^{-1}_mK/Q}\\
& \leq & 0.
\end{eqnarray*}
When $a_m$ is respectively null or infinite, then $\dfrac{\partial \overline{B}_{43}}{\partial {\nu}}$ is respectively null or equal to $-\dfrac{2\nu }{s} \leq 0$ (see \cite{PD} for more details on the behavior of $B_{43}$), which finishes the proof of (\ref{sign}).\\
It is essential that $\overline{B}_{43}$ is concave as the infimum of a family of concave functions, since $\beta$ is convex for all $a>0$.

\item First, let's point out that the function $B_{Q}$ is globally $C^1$ and $C^{2}$ everywhere except on the set $\Pi$. Indeed, to prove this assertion, we refer the reader to the recent paper of Petermichl and Domelevo \cite[Lemma 3]{PD}, where the authors have checked that the first partial derivatives of  $B_{43}$ are continuous throughout three regions, namely $R_1$ where $|\eta|r-\zeta \dfrac{K}{Q}>0$ and $\zeta s-|\eta|\dfrac{K}{Q} >0$, $R_2$ where $|\eta|r-\zeta \dfrac{K}{Q}>0$ and $\zeta s-|\eta|\dfrac{K}{Q} \leq 0$ and $R_3$ where $\zeta s-|\eta|\dfrac{K}{Q} > 0$ and $|\eta|r-\zeta \dfrac{K}{Q} \leq 0$.\\
Now, to prove the derivative property, we will study as before the Hessian of each of $B_{1}$, $B_{2}$, $B_{3}$ and $B_{4}$ separately and then sum the
results to obtain the desired estimate.\newline
\textbf{Case of $B_{1}$:} A direct computation of the Hessian
yields 
\begin{eqnarray*}
-d^{2}B_{1} &=&\frac{2\zeta ^{2}}{r}\left\vert \frac{d\zeta }{\zeta }-\frac{%
dr}{r}\right\vert ^{2}+\frac{2}{s}\langle {d\eta }-\frac{\eta}{s}ds,{d\eta }-\frac{\eta}{s}ds\rangle \\
&\geq  & \frac{4\zeta }{\sqrt{rs}}\left\vert \frac{d\zeta }{\zeta }-%
\frac{dr}{r}\right\vert \left\vert {d\eta }- \frac{\eta}{s}ds%
\right\vert  \\
&\geq & \frac{4 }{Q} \left\vert {d\zeta }-\frac{{\zeta }}{r}dr%
\right\vert \left\vert {d\eta }-\frac{{\eta }}{s}ds\right\vert .
\end{eqnarray*}%
\newline
\textbf{Case of $B_{2}$ and $B_{3}$:} We can deduce from the Hessian of $%
B_{2}$ that of $B_{3}$ simply by replacing the variables $\zeta $ by $\eta $ and $r$ by $s$%
. As in the previous case, the Hessian of the first part of $B_{2}$ is
bounded from below by $\frac{2\zeta ^{2}}{r}\left\vert \frac{d\zeta }{\zeta }%
-\frac{dr}{r}\right\vert ^{2}$. As for the second part, we use the following
lemma:

\begin{lemma}
\label{lemma2} If a function $B$ has the form 
\begin{equation*}
B(F,f,w,M)=F-\frac{f^2}{w+M}
\end{equation*}
and $M$ is a function depending on the variables $w$ and $v$, then if we write $H=B\circ M$, we have
\begin{equation*}
-d^2H \geq -\frac{\partial B}{\partial M}d^2M.
\end{equation*}
\end{lemma}

Cosequently, the Hessian of the second part of $B_{2}$ is bounded from below by $\dfrac{\langle \eta , \eta \rangle}{36Q^{2}s^{2}}r(ds)^{2}$.\newline
Finally, 
\begin{eqnarray*}
-d^{2}B_{2} &\geq & \frac{2\zeta ^{2}}{r}\left\vert \frac{d\zeta }{\zeta }-%
\frac{dr}{r}\right\vert ^{2}+\frac{\langle \eta , \eta \rangle}{36Q^{2}s^{2}}r(ds)^{2} \\
&\geq & \dfrac{\sqrt{2}}{3}\frac{ |\eta | r}{Qrs}|ds|\left\vert {d\zeta }-\frac{\zeta %
}{r}dr\right\vert .
\end{eqnarray*}%
Analogously,%
\begin{equation*}
-d^{2}B_{3}\geq \dfrac{\sqrt{2}}{3}\frac{\zeta  s}{Qrs}|dr| \left\vert d\eta -%
\frac{\eta}{s}ds\right\vert .
\end{equation*}%
\textbf{Case of $B_{4}$:} Once again, we will study separately $B_{41}$, $%
B_{42}$ and $B_{43}$.\newline
We saw in the previous point that when $a_m$ is strictly positive and finite, $K$ is small, id est, $\frac{K}{Q} \leq \frac{|\eta| r}{\zeta}$ and $\frac{K}{Q} \leq \frac{\zeta s}{|\eta|}$. Functions $B_{41}$ and $B_{42}$ are introduced to deal with concavity for other $K$'s. These two functions are treated the same way by replacing $Z,\zeta $ and $r$ repectively by $H,\eta $ and $s$ so we will only focus on $B_{41}$. The idea is to apply chain rule and use Lemma \ref{lemma2} again, knowing that $K\geq Q\frac{|\eta| r}{\zeta }$. \\
More precisely, let 
\[H(\zeta,\eta,r,s)=S(\zeta,\eta,r,s,\widetilde{M}(r,s))=\dfrac{-\zeta^2}{r+\dfrac{\widetilde{M}(r,s)}{Q}}-\dfrac{\langle \eta ,\eta \rangle}{s}.\]
Notice that we omit the variables $Z$ and $H$ because they do not play a role for the Hessian. \\
One checks by calculation of Hessian of $S$ that $-S$ is convex, which means that $-d^2S \geq 0$. By introducing $d_r\widetilde{M}=\dfrac{\partial \widetilde{M} }{\partial r}dr$, $d_s\widetilde{M}=\dfrac{\partial \widetilde{M} }{\partial s}ds$ and applying chain rule we obtain
\begin{eqnarray*}
& &\langle -d^2H (d\zeta,d\eta,dr,ds),(d\zeta,d\eta,dr,ds) \rangle = \\
& &\langle -d^2S(d\zeta ,d\eta ,dr,ds, d_r \widetilde{M}+d_s\widetilde{M}),(d\zeta ,d\eta ,dr,ds,d_r\widetilde{M}+d_s \widetilde{M}) \rangle\\
& & -\dfrac{\partial S}{\partial \widetilde{M}}  \langle d^2\widetilde{M}(dr,ds),(dr,ds)\rangle \label{chainrule}\\
& & \geq -\dfrac{\partial S}{\partial \widetilde{M}}  \langle d^2\widetilde{M}(dr,ds),(dr,ds)\rangle \\
&& \geq \frac{1}{36Q} \frac{\zeta |\eta |}{rs} |ds dr |,
\end{eqnarray*}
because $\dfrac{\partial S}{\partial \widetilde{M}}=\dfrac{\zeta ^2}{Q\left(r+\frac{\widetilde{M}}{Q}\right)^2} \geq \dfrac{\zeta ^2}{36Qr^2}$ and 
$-d^2\widetilde{M} \geq \dfrac{ |ds dr |}{s}$. Moreover, we use the fact that $1 \geq \dfrac{K}{Q} \geq \dfrac{|\eta | r}{\zeta } $.\\
Finally, we obtain the following estimates 
\begin{equation*}
-d^{2}B_{41}\geq \dfrac{1}{36} \frac{\zeta |\eta | }{Qrs}|dsdr|\text{ and }%
-d^{2}B_{42} \geq \dfrac{1}{36} \frac{\zeta |\eta |}{Qrs}|dsdr|.
\end{equation*}%
\newline
We will estimate the Hessian of $B_{43}$ as before by using chain rule to
get rid of the variable $K$. For this purpose, consider $B(\zeta, \eta,r,s,a,K)=\dfrac{-\zeta ^2}{r+a\frac{K}{Q}}+\dfrac{-\langle \eta , \eta \rangle}{s+a^{-1}\frac{K}{Q}}$ and then $H(\zeta, \eta,r,s,K)=B \circ U(\zeta, \eta,r,s,K)$, where $U(\zeta, \eta,r,s,K)=(\zeta, \eta,r,s,a_m(\zeta, \eta,r,s,K),K)$ .\\
As before,
\begin{eqnarray*}
-d^2H=-d^2B + \dfrac{\partial B}{\partial a}| _{a=a_m} \times (-d^2a_m).
\end{eqnarray*}
Recall that $\dfrac{\partial B}{\partial a}=0$ when $a=a_m$. Moreover, by a simple calculation of the Hessian of $B$, we can check that $-d^2B \geq 0$ and so is $-d^2H$. Next, write $L(\zeta, \eta,r,s)=H(\zeta, \eta,r,s,K(r,s))$. Thus,
\begin{eqnarray*}
-d^2L & = & -d^2H + \dfrac{\partial H}{\partial K} \times (-d^2K) \\
& \geq & \frac{1}{Q}\left( \dfrac{a_m\zeta ^2}{(r+a_m\frac{K}{Q})^2}+\dfrac{a_m^{-1}\langle \eta , \eta \rangle }{(s+a_m^{-1}\frac{K}{Q})^2} \right)\times \dfrac{1}{4}|drds| \\
& \geq & \dfrac{1}{4Q}\dfrac{2\zeta |\eta |}{(r+a_m\frac{K}{Q})(s+a_m^{-1}\frac{K}{Q})}|drds| \\
& \geq & \dfrac{1}{8Q}\dfrac{\zeta |\eta |}{rs}|drds|.
\end{eqnarray*}
In the last inequality, we used the fact that when $K$ is small, namely $\dfrac{K}{Q} \leq \dfrac{|\eta |r}{\zeta}$ and $\dfrac{K}{Q} \leq \dfrac{\zeta s}{|\eta |}$, we have $a_m\dfrac{K}{Q} \leq r$ and $a_m ^{-1} \dfrac{K}{Q} \leq s$. The proof of this assertion may be found in \cite{NTV} and \cite{PW}.
\newline
Consequently, we obtain that 
\begin{equation*}
-d^{2}B_{43}\geq  \frac{\zeta |\eta |}{8Qrs}|dsdr|.
\end{equation*}%

\textbf{Conclusion:} In order to finish the proof, it suffices to choose
constants $C_1=1$, $C_2=C_3= \frac{\sqrt2}{3}$ and $C_4=\frac{288}{13}$ in such a way that all terms of $%
-d^{2}B_{Q}=-d^{2}(C_{1}B_{1}+C_{2}B_{2}+C_{3}B_{3}+C_{4}B_{4})$ vanish except
for the term $\dfrac{4}{Q}|d\zeta || d\eta | $.
\end{enumerate}
\end{proof}

As mentioned earlier, $B_{Q}$ fails to be $C^2$ everywhere. We can add smoothness by taking convolutions with mollifiers: for a fixed compact $\mathcal{K}$ in the interior of $\overline{D}_Q$, choose $\varepsilon >0$ such that $\varepsilon <dist(\mathcal{K},\partial \overline{D}_Q)$. Consider $\overline{B}_{\varepsilon,Q}(\overline{\mathcal{X}})=\overline{B}_{Q}*\frac{1}{\varepsilon^{6}}\psi(\frac{\mathcal{\overline{X}}}{\varepsilon})$, where $\psi$ is a bell-shaped infinitely differentiable
function with compact support in the unit ball of $\mathbb{R}^{6}$.\newline
The resulting functions $B_{\varepsilon ,Q}$ and $\overline{B}_{\varepsilon ,Q}$ are clearly smooth and satisfy the following properties
\begin{itemize}
\item[1')] $0 \leq B_{\varepsilon,Q} \leq 80 (1+\varepsilon)(Z+H)$;
\item[2')] $\partial _{\nu}\overline{B}_{\varepsilon ,Q} \leq 0$;
\item[3')] $-d^{2}B_{\varepsilon,Q} \gtrsim \frac{4}{Q}|d\zeta | |d\eta|$ with an implicit constant depending on $\varepsilon$.
\end{itemize}

\begin{proof}
\begin{itemize}
\item[1')] Recall that the original function satisfies $0 \leq B_{Q} \leq 6(Z+H)$. Thus, it is easy to see that size property of the new weighted and mollified function $B_{\varepsilon,Q}$ changes only by a factor depending on the distance between the compact $\mathcal{K}$ and $\partial D_Q$, as well as the sum of weights $C_i$, $i\in \{1, \cdots , 4\}$. 
\item[2')] The non-positivity of $\partial _{\nu}\overline{B}_{\varepsilon ,Q}$ is preserved because the function $\overline{B}_{Q}$ is globally $C^1$ and satisfies 
\[\partial _{\nu}\overline{B}_{\varepsilon ,Q}(\mathcal{\overline{X}}) = \partial _{\nu}\overline{B}_{Q} * \psi _{\varepsilon}(\mathcal{\overline{X}}), \]
for $\mathcal{\overline{X}}=(Z,H,|\zeta| ,\nu ,r,s)$.\\
Since $\partial _{\nu}\overline{B}_{Q}$ is negative and $\psi _{\varepsilon}$ is positive, we obtain the result.
\item[3')] The statement is true because $B_{Q}$ is $C^1$ and we integrate over a compact set. Moreover, the second order derivatives exist almost everywhere (because $\Pi$  is of
measure zero) and are locally integrable, which means they coincide with the second order distributional derivatives. One can find more about this procedure in several previous texts on Bellman functions.
\end{itemize}
\end{proof}

\section{Proof of the main result}

\label{sec_main_result}

This section is devoted to the proof of Theorem \ref{thm1}. \newline
First of all, let's fix $(x,t) \in X \times \R _+$ and define $\widetilde{B} _{\varepsilon ,Q} :\R _+ \times \R _+ \times \R \times T^*X\times\R _+ \times\R _+ \rightarrow \R_+$ such that
\[\widetilde{B} _{\varepsilon ,Q}(Z,H,\zeta ,\eta ,r,s)=B_{\varepsilon ,Q}(Z,H,|\zeta| ,|\eta|_{T^*_xX} ,r,s).\]
Let us also define  for a certain $t_0 >0$ small enough (we will see later how small it should be) the vector
\[v(x,t+t_0)=\left( P_{t+t_0}|f|^{2}\omega_n(x),P_{t+t_0}|\vec{g}|^{2}_{T^*_xX}%
\omega_n^{-1}(x),P_{t+t_0}f(x),\vec{P}_{t+t_0}\vec{g}(x),P_{t+t_0}\omega_n^{-1}(x),P_{t+t_0}%
\omega_n(x)\right) \]
and in parallel
\[\overline{v}(x,t+t_0)=\left( P_{t+t_0}|f|^{2}\omega_n(x),P_{t+t_0}|\vec{g}|^{2}_{T^*_xX}%
\omega_n^{-1}(x),P_{t+t_0}f(x),|\vec{P}_{t+t_0}\vec{g}(x)|_{T^*_xX},P_{t+t_0}\omega_n^{-1}(x),P_{t+t_0}%
\omega_n(x)\right) \]
where we recall that $P_t$ and $\vec{P}_{t}$ stand
for the weighted Poisson extensions. It is important to mention that $v(x,t+t_0) \in D_Q$ and $(x,t)\mapsto v(x,t)$ maps compacts in $X \times \R_+$ to compacts in $D_Q$. Indeed, the following inequalities
\[|P_{t+t_0}f|^2 \leq P_{t+t_0}\left( |f|^{2}\omega_n \right) P_{t+t_0}\omega_n^{-1} \text{ and } |\vec{P}_{t+t_0}\vec{g}|^2_{T^*_xX} \leq P_{t+t_0}\left(|\vec{g}|^{2}_{T^*_xX}  \omega_n^{-1}\right)P_{t+t_0}\omega_n \]
are true by Lemma \ref{lemma1}. It is also clear that $P_{t+t_0}\omega_n(x)P_{t+t_0}\omega_n^{-1}(x)\leq Q$ by the very definition of $Q$. It remains to show that it is greater than $1$.\\
Since the semigroup $P_t$ is Markovian, $P_t 1 = 1$ by \cite{Bakry}. Thus
\begin{eqnarray*}
1 &=&  \int_X p_{t+t_0}^{1/2}(x,y)\omega_n^{1/2}(y)p_{t+t_0}^{1/2}(x,y)\omega_n^{-1/2}(y)d \mu _{\varphi}(y) \\
& \leq & \left( P_{t+t_0} \omega_n (x) \right) ^{1/2} \times \left( P_{t+t_0}\omega _n^{-1} (x) \right) ^{1/2}.
\end{eqnarray*}
Besides, the mapping property holds because $v$ is a continuous function. \\
Next, define 
\[b_{\varepsilon}(x,t+t_0)=\widetilde{B} _{\varepsilon ,Q} (v(x,t+t_0))\]
and consider the operators%
\begin{equation*}
\Delta_{\varphi,t}=\partial^2_{tt}+\Delta_{\varphi} \text{ and } \vec{\Delta}_{\varphi,t}%
=\partial^2_{tt}+\vec{\Delta_{\varphi}}.
\end{equation*}
The fact that $\widetilde{B} _{\varepsilon ,Q} $ is radial allows us to define the Bellman function on manifolds. Our goal is to find a link between $\Delta_{\varphi,t}b_{\varepsilon}$ and $d^2\widetilde{B} _{\varepsilon ,Q} $ and then
estimate the integral%
\begin{equation*}
-\int \int \Delta_{\varphi,t}b_{\varepsilon}(x,t+t_0) d\mu_{\varphi}(x)t d t
\end{equation*}%
from below and above.

\subsection{Estimate from below}

%
%

\begin{proposition} 
\label{below} Suppose that $\Ric _{\varphi} \geq 0$. Then for all $(x,t)\in X\times \R_+$,
\begin{equation*}
-\Delta_{\varphi,t}b_{\varepsilon}(x,t+t_0)\geq \frac{4}{Q}|\overline{\nabla} P_{t+t_0}f(x)||\overline{%
\nabla} \vec{P}_{t+t_0}\vec{g}(x)|.
\end{equation*}
\end{proposition}

\begin{proof}
Following \cite[Lemma 12]{DC} and \cite[Proposition 13]{DC} we use the function $B_{\varepsilon,Q}$ and the corresponding function $\overline{B}_{\varepsilon,Q}$ to define the quantity
\begin{eqnarray*}
F(x,t+t_0)&=&-\Delta_{\varphi,t} \widetilde{B} _{\varepsilon ,Q} (v(x,t+t_0)) \\
&& - \frac{\partial_{\nu}\overline{B} _{\varepsilon ,Q}(\overline{v}(x,t+t_0))}{|\vec{P}_t\vec{g}|_{T^*_xX}}\times (-\Ric_{\varphi}(\sharp \vec{P}_t\vec{g},\sharp \vec{P}_t\vec{g})),
\end{eqnarray*}
where $\sharp :T^*_xX \rightarrow T_xX$ is the sharp musical isomorphism. \\
Note that we have a different sign convention for $\Delta_{\varphi,t}$ and use concavity instead of convexity.\\
Since $-d^2 B _{\varepsilon,Q} \gtrsim \frac{4}{Q}|d\eta||d\zeta|$ , this implies
\[F(x,t+t_0) \gtrsim \frac{4}{Q} |\overline{\nabla}P_tf||\overline{\nabla}\vec{P}_t\vec{g}|.\]
Furthermore, since $\partial_{\nu}\overline{B}_{\varepsilon,Q} \leq 0$, we have
\[-\Delta_{\varphi,t} \widetilde{B} _{\varepsilon ,Q} (v(x,t+t_0)) \geq \frac{4}{Q} |\overline{\nabla}P_tf||\overline{\nabla}\vec{P}_t\vec{g}| .\]
The calculations used to compute $F$ are omitted since they follow exactly the same steps as in \cite{DC}, that is, computing $\Delta_{\varphi,t}\widetilde{B} _{\varepsilon ,Q} $ and writing $F$ in terms of its variables and their different derivatives by using the Bochner formula in \cite[eq. (0.3)]{bak}. To verify the desired inequality, one can use exponential local coordinates and inequality (\ref{deriv}), since the expression of $F$ holds pointwise.
\end{proof}

\subsection{Estimate from above}

In order to prove Theorem \ref{thm1}, it suffices to show that 
\begin{equation*}
-\int \int \Delta_{\varphi,t}b_{\varepsilon}(x,t+t_0) d\mu_{\varphi}(x)t d t \leq
C  \|f\|_{L^{2}(X,\omega_n\mu_{\varphi})} \|\vec{g}\|_{L^{2}(T^{*}X,\omega_n^{-1}%
\mu_{\varphi})}.
\end{equation*}
In fact, for a fixed point $o\in X$, $l>1$ and $s>0$ such that $t_0 \in (0,1/s)$, we have the following result:
\begin{proposition}
\label{above} Suppose that $\Ric_{\varphi}\geq 0$. Then for every $(x,t)$ in the compact set $K_{s,l}:= \overline{B(o,l)} \times [1/s,s]$ we have
\begin{equation*}
\varlimsup_{s \to \infty} \varlimsup_{l \to \infty}\int_{1/s}^s \int_{B(o,l)}
-\Delta_{\varphi,t}b_{\varepsilon}(x,t+t_0) d\mu_{\varphi}(x)t d t \leq
20(1+\varepsilon)(\|f\|^2_{L^{2}(X,\omega_n\mu_{\varphi})}+ \|\vec{g}%
\|^2_{L^{2}(T^{*}X,\omega^{-1}_n\mu_{\varphi})}).
\end{equation*}
\end{proposition}

\begin{proof}
Let $\rho (x,o)$ be the geodesic distance on $X$ between $o$ and $x$. Define $\Lambda \in C_c^{\infty}([0,\infty))$ be a decreasing function such that $0\leq \Lambda \leq 1$, $\Lambda =1$
in $[0,1]$ and $\Lambda =0$ in $[2,\infty)$. We are interested in the following composite function :
\begin{equation*}
\Lambda \left( \frac{\rho (x,o)^2}{l^2} \right), l>1.
\end{equation*}
Observe that this composite function is always positive and equals to $1$ when $\rho (x,o) <l$. Recall also that by Proposition \ref{below}, $-\Delta_{\varphi,t}b_{\varepsilon} \geq 0$ and so 
\[\int_{1/s}^s \int_{B(o,l)}
-\Delta_{\varphi,t}b_{\varepsilon}(x,t+t_0) d\mu_{\varphi}(x)t d t \leq \int_{1/s}^s \int_{X}
-\Delta_{\varphi,t}b_{\varepsilon}(x,t+t_0)\Lambda \left( \frac{\rho (x,o)^2}{l^2} \right) d\mu_{\varphi}(x)t d t.\]
To prove the lemma, we shall show that
\begin{eqnarray}
&& \varlimsup_{s \to \infty} \varlimsup_{l \to \infty}\int_{1/s}^s \int_{X} -\partial _{tt}^{2}b_{\varepsilon}(x,t+t_0)\Lambda \left( \frac{\rho (x,o)^2}{l^2} \right) d\mu_{\varphi}(x)t \deriv t   \notag \\ 
& \leq &20(1+\varepsilon)(\|f\|^2_{L^{2}(X,\omega_n\mu_{\varphi})}+ \|\vec{g}\|^2_{L^{2}(T^{*}X,\omega_n^{-1}\mu_{\varphi})}),\label{10}
\end{eqnarray}
and
\begin{equation}
\lim_{l \to \infty}\int_{1/s}^s \int_{X} -\Delta_{\varphi}b_{\varepsilon}(x,t+t_0)\Lambda \left( \frac{\rho (x,o)^2}{l^2} \right) d\mu_{\varphi}(x)t \deriv t =0. \label{11}
\end{equation}
We first prove (\ref{10}). An integration by parts in the variable $t$ gives%
\[
\int_{1/s}^{s}-\partial _{tt}^{2}b_{\varepsilon}\left( x,t+t_0\right) tdt=\frac{1}{s}\partial _{t}b_{\varepsilon}\left( x,\frac{1}{s}+t_0%
\right)-s\partial_{t}b_{\varepsilon}\left( x,s+t_0\right)+b_{\varepsilon}\left( x,s+t_0\right)-b_{\varepsilon}\left( x,\frac{1}{s}+t_0\right).\]
The size property (\ref{size}) implies
\begin{eqnarray*}
b_{\varepsilon}\left( x,s+t_0\right)-b_{\varepsilon}\left( x,\frac{1}{s}+t_0\right) &\leq
&b_{\varepsilon}\left( x,s+t_0\right) \\
&\leq &20(1+\varepsilon)\left(
P_{s+t_0}(\left\vert f \right\vert
^{2}\omega_n)(x)+P_{s+t_0}(\left\vert \vec{g}  \right\vert ^{2}_{T^*_xX}\omega_n^{-1})(x)\right) .
\end{eqnarray*}
It follows by contractivity of the semigroup $P_t$ on $L^{r}(\mu _{\varphi })$ for every $r\in \left[1,+\infty \right]$ (read \cite{Strich} as a reference) that 
\begin{eqnarray*}
\int_{\frac{1}{s}}^{s}\int_{X}-\partial _{tt}^{2}b_{\varepsilon}\left( x,t+t_0\right) \Lambda \left( \frac{\rho (x,o)^2}{l^2} \right)
 d\mu _{\varphi }\left( x\right) tdt &\leq &20(1+\varepsilon)\left( \left\Vert
f\right\Vert ^{2}_{L^{2}(X,\omega_n\mu_{\varphi})}+\left\Vert \vec{g} \right\Vert ^{2}_{L^{2}(T^*X,\omega_n^{-1}\mu_{\varphi})}\right) \\
&& + \left\Vert
s\partial _{t}b_{\varepsilon}\left( \cdot,s+t_0\right) -\frac{1}{s}\partial _{t}b_{\varepsilon}\left( \cdot,%
\frac{1}{s}+t_0\right) \right\Vert _{L^{1}\left( \mu _{\varphi }\right) },
\end{eqnarray*}
Therefore, in order to prove (\ref{10}) it is enough to show that%
\begin{equation}\label{12}
\lim_{s\rightarrow +\infty }\left\Vert s\partial _{t}b_{\varepsilon}\left(
\cdot,s+t_0\right)\right\Vert _{L^{1}\left( \mu _{\varphi }\right) }=0,  
\end{equation}
and
\begin{equation}\label{add}
\lim_{s\rightarrow +\infty }\left\Vert \frac{1}{s}\partial _{t}b_{\varepsilon}\left(
\cdot,\frac{1}{s}+t_0\right)\right\Vert _{L^{1}\left( \mu _{\varphi }\right) }=0,  
\end{equation}
By applying chain rule we obtain
\begin{eqnarray*}
\partial _{t}b_{\varepsilon}\left(x,s+t_0\right) &=&  \frac{\partial \widetilde{B} _{\varepsilon ,Q} }{\partial Z}(v)\partial_t P_{s+t_0}|f|^2\omega_n (x) + \frac{\partial \widetilde{B} _{\varepsilon ,Q} }{\partial H}(v)\partial_t P_{s+t_0}|\vec{g}|_{T^*_xX}^2\omega_n ^{-1}(x)\\
&& + \frac{\partial \widetilde{B} _{\varepsilon ,Q} }{\partial \zeta}(v)\partial_t P_{s+t_0}f (x) + \langle \frac{\partial \widetilde{B} _{\varepsilon ,Q} }{\partial \eta}(v),\partial_t \vec{P}_{s+t_0}\vec{g} (x)\rangle_{T^*_xX} \\
&& + \frac{\partial \widetilde{B} _{\varepsilon ,Q} }{\partial r}(v)\partial_t P_{s+t_0}\omega_n (x) + \frac{\partial \widetilde{B} _{\varepsilon ,Q} }{\partial s}(v)\partial_t P_{s+t_0}\omega_n ^{-1}(x).
\end{eqnarray*}
Thus, using H\"{o}lder's inequality with some $\alpha$ and its conjugate exponent $\alpha^{\prime }$ we obtain
\begin{eqnarray} \label{sigma}
\left\Vert s\partial _{t}b_{\varepsilon}\left(
\cdot,s+t_0\right)\right\Vert _{L^{1}\left( \mu _{\varphi }\right) } &\leq & \underbrace{\left\Vert  \vert \frac{\partial \widetilde{B} _{\varepsilon ,Q} }{\partial Z}(v)\vert + \cdots + \vert \frac{\partial \widetilde{B} _{\varepsilon ,Q} }{\partial s}(v)\vert \right\Vert_{L^{\alpha}(\mu _{\varphi })}}_{\text{$\Sigma_1$}} \\
&& \times \underbrace{\left\Vert  s \left( \vert \partial _t P_{s+t_0}\left(|f|^2\omega_n\right)(\cdot)\vert + \cdots + \vert \partial _t P_{s+t_0}\omega_n ^{-1}(\cdot)\vert \right) \right\Vert _{L^{\alpha^{\prime}}(\mu _{\varphi})}}_{\text{$\Sigma_2$}} \notag.
\end{eqnarray}

Let's study $\Sigma_1$. First of all, by triangle inequality, one sees that $\Sigma_1$ can be majorized by the sum of norms of each partial derivative of $\widetilde{B} _{\varepsilon ,Q}$. These partial derivatives show terms in $P_t$ and $\vec{P}_t$. We will then need to estimate the obtained norms uniformly in $t>0$ so that we can tend $s$ to $+\infty$ in (\ref{12}). This can be done by using H\"{o}lder's inequality and contractivity of both $P_t$ and $\vec{P}_t$.\\
For instance, we already know from Section \ref{sec_bell} (proof of Lemma \ref{lemma3}, 2)) that when $a_m$ is strictly finite and positive,
\[\langle\frac {\partial B_{Q ,\mathcal{K}}}{\partial \eta},d\eta \rangle =-\frac{6\langle \eta, d\eta \rangle}{s}-\frac{2\langle \eta, d\eta \rangle}{s+\frac{M}{Q^2}}-\frac{2\langle \eta, d\eta \rangle}{s+\frac{\widetilde{N}}{Q}} -\frac{2\langle \eta, d\eta \rangle}{s+a_m^{-1}\frac{K}{Q}} , \text{ for all $d\eta$, }\]
which implies that
\[ \vert \frac {\partial B_{Q ,\mathcal{K}}}{\partial \eta}\vert \leq \vert\frac{6 \eta }{s} \vert + \vert\frac{2 \eta}{s+\frac{M}{Q^2}} \vert + \vert\frac{2 \eta}{s+\frac{\widetilde{N}}{Q}} \vert+\vert\frac{2 \eta}{s+a_m^{-1}\frac{K}{Q}} \vert.\]
Recall that $1\leq rs $ and that $M$, $\widetilde{N}$ and $a_m^{-1}\frac{K}{Q}$ are positive. Thus, we have $\vert\dfrac {\partial B_{Q ,\mathcal{K}}}{\partial \eta}\vert \leq 12 \vert \eta  \vert r$. Since $B_{Q ,\mathcal{K}}$ is $C^1$, we can also dominate $\vert\dfrac {\partial B_{\varepsilon, Q}}{\partial \eta}\vert$, but with a constant depending on $\varepsilon$. Meaning that there exists a constant $C_{\varepsilon} > 0$ such that
\[\vert\dfrac {\partial B_{\varepsilon, Q}}{\partial \eta}(\mathcal{X})\vert  \leq C_{\varepsilon}|\eta |r.\]
Finally, by replacing the variable $\mathcal{X}$ by $v(x,t+t_0)$ we obtain
\[\left\Vert  \frac {\partial \widetilde{B} _{\varepsilon ,Q} }{\partial \eta}\right \Vert_{L^{\alpha}(\mu _{\varphi})} \leq C_{\varepsilon}\left\Vert P_{t+t_0} \omega_n P_{t+t_0} |\vec{g}|_{T^*_xX}\right\Vert_{L^{\alpha}(\mu _{\varphi})},\]
and by symmetry,
\[\left\Vert  \frac {\partial \widetilde{B} _{\varepsilon ,Q} }{\partial \zeta}\right \Vert_{L^{\alpha}(\mu _{\varphi})} \leq C_{\varepsilon}\left\Vert P_{t+t_0} \omega_n ^{-1} {P}_{t+t_0}{f}\right\Vert_{L^{\alpha}(\mu _{\varphi})}.\]
The result is the same up to a constant when $a_m$ is null or infinite. Using the same arguments as above, we can dominate the first partial derivatives in the other variables. Indeed,
\[\left\Vert  \frac {\partial \widetilde{B} _{\varepsilon ,Q} }{\partial Z}\right \Vert_{L^{\alpha}(\mu _{\varphi})} = \left\Vert  \frac {\partial \widetilde{B} _{\varepsilon ,Q} }{\partial H}\right \Vert_{L^{\alpha}(\mu _{\varphi})} \leq C_{\varepsilon}.\]
Moreover, if $a_m$ is finite and positive, we have
\begin{eqnarray*}
\frac {\partial {B} _{Q, \mathcal{K}} }{\partial r} &=& \dfrac{3 \zeta ^2}{r ^2}+\dfrac{Q^2\langle \eta ,\eta\rangle\dfrac{\partial M(r,s)}{\partial r}}{(Q^2s+M(r,s))^2} + \dfrac{Q^2\zeta ^2 (Q^2+\dfrac{\partial N(r,s)}{\partial r})}{(Q^2r+N(r,s))^2} \\
& & + \dfrac{Q\zeta ^2 (Q+\dfrac{\partial \widetilde{M}(r,s)}{\partial r})}{(Qr+\widetilde{M}(r,s))^2} 
 + \dfrac{Q\langle \eta ,\eta \rangle \dfrac{\partial \widetilde{N}(r,s)}{\partial r}}{(Qs+\widetilde{N}(r,s))^2} \\
& & +\left( \dfrac{\zeta ^2(1+\dfrac{a_m }{Q}\dfrac{\partial K(r,s)}{\partial r})}{(r+a_m\dfrac{K(r,s)}{Q})^2}+\dfrac{\langle \eta ,\eta \rangle \dfrac{a_m^{-1}}{Q}\dfrac{\partial K(r,s)}{\partial r}}{(s+a_m ^{-1}\dfrac{K(r,s)}{Q})^2}\right),
\end{eqnarray*}
where the last derivative between brackets represents $\partial _{r}B_{43}$ and has been calculated as $\partial _{\eta}B_{43}$ in Section \ref{sec_bell}.\\
The partial derivatives in $r$ of $M$, $N$, $K$, $\widetilde{M}$ and $\widetilde{N}$ are
\begin{eqnarray*}
\dfrac{\partial M(r,s)}{\partial r} &=& \dfrac{4Q^2}{r^2}-s^2 \text{,  } \dfrac{\partial N(r,s)}{\partial r} = -2sr + (4Q^2+1),\\
\dfrac{\partial K(r,s)}{\partial r} &=& \dfrac{\sqrt{Q}}{2}\sqrt{\dfrac{s}{r}}-\dfrac{s}{4} \text{,  } \\
 \dfrac{\partial \widetilde{M}(r,s)}{\partial r} & = & -\dfrac{rs}{2Q}+(4Q+1)  \text{  and  } 
\dfrac{\partial \widetilde{N}(r,s)}{\partial r} = \dfrac{4Q}{r^2}-\dfrac{s^2}{4Q}.
\end{eqnarray*}
Thus, using that $ rs \geq 1$, we obtain
\begin{eqnarray*}
\vert \frac {\partial {B} _{Q, \mathcal{K}} }{\partial r} \vert & \leq & {3\zeta ^2s^2}+{\langle \eta ,\eta \rangle }+8\zeta ^2s^2+7\zeta ^2s^2\\
&& +5 \langle \eta ,\eta \rangle  + \zeta ^2s^2(1+\dfrac{a_m }{Q}\dfrac{\partial K(r,s)}{\partial r}) \\
&&+ \langle \eta ,\eta \rangle r^2(1+\dfrac{a_m ^{-1}}{Q}\dfrac{\partial K(r,s)}{\partial r}).
\end{eqnarray*}
The next step is to bound $a_m \dfrac{\partial K(r,s)}{\partial r}$ and $a_m^{-1} \dfrac{\partial  K(r,s)}{\partial r} $ from above. In fact, since $\dfrac{\partial K(r,s)}{\partial r}=\dfrac{\sqrt{Q}}{2}\sqrt{\dfrac{s}{r}}-\dfrac{s}{4}$, one can observe that $r\dfrac{\partial K(r,s)}{\partial r} \leq K(r,s)$. In addition, $rs \geq 1$ and $a_m\frac{K}{Q} \leq r$ so we can write
\begin{eqnarray*}
a_m \dfrac{\partial K(r,s)}{\partial r} & \leq & a_m sr \dfrac{\partial K(r,s)}{\partial r} \\
& \leq & s a_m K(r,s) \\
& \leq & Qsr \\
& \leq & Q^2.
\end{eqnarray*}
Likewise, $a_m^{-1} \dfrac{\partial  K(r,s)}{\partial r} \leq Qs^2$ and finally we obtain
\[\vert \frac {\partial {B} _{Q, \mathcal{K}} }{\partial r} \vert  \leq C \left( \zeta ^2 s^2 + \langle\eta , \eta\rangle r^2  +\langle\eta , \eta\rangle \right).\]
Therefore,
\begin{eqnarray*}
\left\Vert  \frac {\partial \widetilde{B} _{\varepsilon ,Q} }{\partial r}\right \Vert_{L^{\alpha}(\mu _{\varphi})} &\leq&  C_{\varepsilon} \left( \left\Vert P_{t+t_0} f {P}_{t+t_0} \omega_n ^{-1} \right\Vert ^2_{L^{2\alpha}(\mu _{\varphi})} +  \right. \\
&&  \left.  \left\Vert {P}_{t+t_0} |\vec{g}|_{T^*_xX} P_{t+t_0} \omega_n \right\Vert ^2_{L^{2\alpha}(\mu _{\varphi})} + \left\Vert \vec{P}_{t+t_0} \vec{g}\right\Vert ^2_{L^{2\alpha}(\mu _{\varphi})} \right)
\end{eqnarray*}
and again by symmetry,
\begin{eqnarray*}
\left\Vert  \frac {\partial \widetilde{B} _{\varepsilon ,Q} }{\partial s}\right \Vert_{L^{\alpha}(\mu _{\varphi})} &\leq & C_{\varepsilon} \left( \left\Vert {P}_{t+t_0} |\vec{g}|_{T^*_xX} P_{t+t_0} \omega_n \right\Vert ^2_{L^{2\alpha}(\mu _{\varphi})}+  \right. \\
&& \left.  \left\Vert P_{t+t_0}f P_{t+t_0} \omega_n ^{-1}\right\Vert ^2_{L^{2\alpha}(\mu _{\varphi})}+\left\Vert P_{t+t_0}f \right\Vert ^2_{L^{2\alpha}(\mu _{\varphi})} \right).
\end{eqnarray*}
Now, if $a_m$ is null or infinite, then $\dfrac{\partial B_{43}}{\partial r}$ is either $-\dfrac{|\zeta|^2}{r^2}$ or $0$. We repeat the previous calculations and obtain the same results up to a constant.\\
As said before, we now need to estimate from above these norms for each $i=1,\dots,6$ uniformly in $t>0$. To do so, we use H\"{o}lder's inequality and contractivity of both $P_t$ and $\vec{P}_t$ in $L^{r}(\mu _{\varphi})$ for all $r\in \left[ 1,+\infty \right] $. In other terms, we have shown that $\Sigma_1$ appearing in (\ref{sigma}) can be majorized in the following way
\[
\Sigma_1 \leq C\left( \varepsilon,f,\vec{g},\omega,\omega^{-1}\right) 
\]
uniformly in $t>0$.

As for part $\Sigma_2$, all we have to do is to show that the quantity%
\begin{equation*}
\left\Vert s\left( \left\vert \partial_{t}P_{s+t_0}\left(\vert f\vert^2\omega_n \right)\right\vert +\left\vert \partial_{t}P_{s+t_0}\left( \vert \vec{g}\vert^2_{T^*_xX}\omega_n^{-1} \right) \right\vert + \cdots +\left\vert \partial _{t}P_{s+t_0}\omega_n^{-1} \right\vert \right) \right\Vert _{L^{\alpha ^{\prime }}(\mu _{\varphi})}
\end{equation*}
tends to 0 as $s$ tends to infinity. In fact, simply observe that by the Hilbert space spectral respresentation theory, each of $s\partial _{t}P_{s+t_0}f$, $s\partial_t\vec{P}_{s+t_0}\vec{g} $, $s\partial_tP_{s+t_0}\left( \vert f\vert^2\omega_n \right) $, $s\partial_tP_{s+t_0}\left( \vert\vec{g}\vert^2_{T^*_xX}\omega_n^{-1}\right) $, $s\partial _{t}P_{s+t_0}\omega_n$ and $s\partial _{t}P_{s+t_0}\omega_n^{-1}$ converge to $0$ in $L^{2}(\mu _{\varphi})$ as $s$ tends to infinity, because functions $P_{s+t_0}f$, $\vec{P}_{s+t_0}\vec{g}$, $P_{s+t_0}\omega_n$ and $P_{s+t_0}\omega_n ^{-1}$ are square integrable.\\
To conclude, notice that $\left\Vert t\partial
_{t}P_{t}\right\Vert _{L^{r}\left(X, \mu _{\varphi }\right)}+\left\Vert t\partial _{t}\vec{P}%
_{t}\right\Vert _{L^{r}\left(X, \mu _{\varphi }\right)}$ is uniformly bounded in $t$ for all $r$ in $\left(
1,\infty \right) $ \cite[Theorem 4.6 (c)]{11}, because $P_{t}$ and $\vec{P}%
_{t}$ are symmetric Markov semigroups  respectively on $L^{2}\left(X, \mu _{\varphi }\right) $ and $L^{2}\left(T^*X, \mu _{\varphi }\right) $ and thus extend to bounded holomorphic semigroups respectively on $L^{r}\left(X, \mu _{\varphi }\right) $ and $L^{r}\left(T^*X, \mu _{\varphi }\right) $, for
all $r$ in $\left( 1,\infty \right) $ \cite[Theorem 1.4.2]{Dav}.\\
We follow the same procedure to prove (\ref{add}).\\

We now prove (\ref{11}). By integrating by parts twice, we have
\[-\int_{\frac{1}{s}}^{s}\int_{X}\Delta_{\varphi }b_{\varepsilon}\left( x,t+t_0\right)\Lambda \left( \frac{\rho (x,o)^2}{l^2} \right) d\mu _{\varphi }\left( x\right)
tdt\]
\begin{flushright}
 \[= - \int_{\frac{1}{s}}^{s}\int_{X}b_{\varepsilon}\left( x,t+t_0\right)
\Delta_{\varphi }\Lambda \left( \frac{\rho (x,o)^2}{l^2} \right) d\mu _{\varphi }\left( x\right)
tdt.\]
\end{flushright}
A simple computation based on \cite[p 140]{bak} gives%
\begin{eqnarray*}
-\Delta_{\varphi }\Lambda \left( \frac{\rho (x,o)^2}{l^2} \right) &=& -\frac{2}{l^{2}}\left( |d\rho (x,o)|^2 + \rho (x,o) 
\Delta_{\varphi }\rho (x,o) \right)\Lambda ^{\prime }\left( 
\frac{\rho ^2 (x,o)}{l^{2}}\right) \\
&&-\frac{4\rho ^2 (x,o)}{l^{4}}\left\vert d\rho (x,o)
\right\vert^2 \Lambda^{\prime \prime}\left( \frac{\rho ^2 (x,o) }{l^{2}}\right) ,
\end{eqnarray*}
for all $x\in X\backslash (\cut(o) \cup \lbrace o\rbrace)$, where $\cut(o)$ denotes the cut locus of the point $o$. In addition, since $\Ric_{\varphi }\geq 0$, by \cite[Theorem 1.1]{23} we have the local comparison result%
\begin{equation}
\Delta_{\varphi }\rho (x,o) \leq C\frac{1}{\rho (x,o)} ,  \label{15}
\end{equation}
for all $x\in X\backslash (\cut(o) \cup \lbrace o\rbrace)$.\\
Since $\left\Vert d \rho \right\Vert _{\infty }\leq 1$, and $\supp \Lambda$ is in $[0,2]$, by (\ref{15}) there exists $C>0$ such that%
\begin{eqnarray}
-\Delta_{\varphi }\Lambda \left( \frac{\rho (x,o)^2}{l^2} \right) &\geq & -C\left( \left\Vert \Lambda
^{\prime }\right\Vert _{\infty }+\left\Vert \Lambda ^{\prime \prime}\right\Vert _{\infty
}\right)\chi_{B(o,2l) \backslash B(o,l)}, 
\label{16}
\end{eqnarray}%
for all $x\in X\backslash (\cut(o) \cup \lbrace o\rbrace)$ and provided $l\geq 1$. Moreover, (\ref{16}) holds weakly on $X$ and in particular, we have
\[
\int_{X}-\Delta_{\varphi }b_{\varepsilon}\left( x,t+t_0\right) \Lambda \left( \frac{\rho (x,o)^2}{l^2} \right)
d\mu _{\varphi }\left( x\right) \geq -C\int_{B\left( o,2l\right) \backslash
B\left( o,l\right) }b_{\varepsilon}\left( x,t+t_0\right) d\mu _{\varphi }\left( x\right) 
.\]%
Hence, the size property (\ref{size}) implies that%
\begin{eqnarray*}
&& \int_{X}-\Delta_{\varphi }b_{\varepsilon}\left( x,t+t_0\right) \Lambda \left( \frac{\rho (x,o)^2}{l^2} \right)
d\mu _{\varphi }\left( x\right) \\
&& \geq  -C\int_{B\left( o,2l\right) \backslash
B\left( o,l\right) }\left( P_{t+t_0}\left(\left\vert f\right\vert
^{2}\omega_n \right) (x)+P_{t+t_0} \left( \left\vert \vec{g} \right\vert_{T^*_xX} ^{2}\omega_n ^{-1} \right) (x) \right) d\mu _{\varphi }(x).
\end{eqnarray*}
Denote the integral on the right-hand side by $\Psi _{l}\left( t\right)$. Since $\lim_{l\rightarrow +\infty }\Psi _{l}=0$ pointwise on $\mathbb{R}_{+}$
and $0\leq \Psi _{l}\left( t\right) \leq \|f\|^2_{L^{2}(X,\omega_n\mu_{\varphi})}+ \|\vec{g}%
\|^2_{L^{2}(T^{*}X,\omega_n^{-1}\mu_{\varphi})},$ the Lebesgue dominated
convergence theorem implies%
\begin{equation}
\liminf_{l\rightarrow +\infty }\int_{\frac{1}{s}}^{s}\int_{X}-\Delta_{\varphi }b_{\varepsilon}\left( x,t+t_0\right) \Lambda \left( \frac{\rho (x,o)^2}{l^2} \right) d\mu _{\varphi }\left(
x\right) \geq 0.  \label{Tun1}
\end{equation}%
It remains to prove that 
\[
\limsup_{l\rightarrow +\infty }\int_{\frac{1}{s}}^{s}\int_{X}-\Delta_{\varphi }b_{\varepsilon}\left( x,t+t_0\right) \Lambda \left( \frac{\rho (x,o)^2}{l^2} \right) d\mu _{\varphi }\left(
x\right) tdt\leq 0.
\]%
Consider the function 
\[
R(x,t+t_0)=20 \left( 1+\varepsilon \right) 
\left( P_{t+t_0}\left(\left\vert f\right\vert ^{2}\omega_n \right) (x) +P_{t+t_0}\left( \left\vert \vec{g}
\right\vert_{T^*_xX} ^{2}\omega_n^{-1}\right) (x)\right).  
\]%
We have $b_{\varepsilon}-R\leq 0$ on $K_{s,l}$ and by an argument similar to the
one we used to prove (\ref{Tun1}) one shows that 
\[
\limsup_{l\rightarrow +\infty }\int_{\frac{1}{s}}^{s}\int_{X}-\Delta_{\varphi }\left(
b_{\varepsilon}\left( x,t+t_0\right) -R(x,t+t_0)\right) \Lambda \left( \frac{\rho (x,o)^2}{l^2} \right) d\mu
_{\varphi }\left( x\right) tdt\leq 0.
\]%
It suffices then to prove that%
\begin{equation}
\limsup_{l\rightarrow +\infty }\int_{1/s}^{s}\int_{X}\Delta_{\varphi }R\left( x,t+t_0\right) \Lambda \left( \frac{\rho (x,o)^2}{l^2} \right) d\mu _{\varphi }\left( x\right)
tdt=0,  \label{Tun2}
\end{equation}%
using an integration by parts. To this purpose, notice first that the composite function $\Lambda \left( \dfrac{\rho ^2}{l^2} \right)$ is equal to zero for $\rho \geq 2l$ and hence we have
\[ \|d\Lambda \left( \frac{\rho ^2}{l^2} \right)\|_{\infty} \leq \frac{4\|\Lambda '\|_{\infty}}{l} .\] Then, 
\begin{eqnarray} \label{Tun3}
\left\vert \int_{\frac{1}{s}}^{s}\int_{X}\Delta_{\varphi }R\left( x,t+t_0\right)
\Lambda \left( \frac{\rho (x,o)^2}{l^2} \right) d\mu _{\varphi }\left( x\right) tdt\right\vert \leq \frac{4\left\Vert \Lambda '\right\Vert _{\infty }}{l}%
\int_{\frac{1}{s}}^{s}\int_{X}\left\vert dR\left( x,t+t_0\right) \right\vert d\mu
_{\varphi }\left( x\right) tdt.
\end{eqnarray}%
Now, notice that by Lemma \ref{lemma1} and semigroup contractivity we have
\begin{eqnarray*}
\int_{X} |dR(x,t+t_0)|  d\mu_{\varphi } (x)  &\leq & C(\varepsilon) \left( \int_{X}  | \vec{P_t}\left( dP_{t_0}( \left\vert f\right\vert ^{2}\omega_n )\right) | (x) \right. \\
&& + \left. | \vec{P_t} \left( d P_{t_0}( \left\vert {\vec{g}} \right\vert _{T^*_xX} ^{2}\omega^{-1}_n) \right)|  (x)  d\mu
_{\varphi } (x) \right)  \\
& \leq & C(\varepsilon)  \int_{X}  | dP_{t_0}( \left\vert f\right\vert ^{2}\omega_n ) | (x) + | d P_{t_0}( \left\vert {\vec{g}} \right\vert_{T^*_xX} ^{2}\omega^{-1}_n) |  (x)  d\mu
_{\varphi } (x)  \\
& \leq & C(\varepsilon)\left[ \left( \int_{X} |dP_{t_0}(|f|^2\omega_n)|^2 d\mu
_{\varphi } (x) \right)^{1/2} \times \left( \int_{X}  d\mu
_{\varphi } (x) \right)^{1/2} \right. \\
&& + \left. \left( \int_{X} |dP_{t_0}(|\vec{g}|_{T^*_xX}^2\omega ^{-1}_n)|^2  d\mu
_{\varphi } (x) \right)^{1/2} \times \left( \int_{X}  d\mu
_{\varphi }(x) \right)^{1/2} \right]
\end{eqnarray*}
Now observe that by using (\ref{page4}) and because $X$ is without a boundary we have
\begin{eqnarray*}
\int_{X} | dP_{t_0}(|f|^2\omega_n)(x)|^2 d\mu
_{\varphi }(x) & \leq & \left( \int_{X} |P_{t_0} (|f|^2\omega_n)(x)|^2 d\mu_{\varphi } (x)\right)^{1/2} \\ && \times \left( \int_{X} |\Delta_{\varphi}P_{t_0} (|f|^2\omega_n)(x)|^2 d\mu_{\varphi } (x)\right)^{1/2} \\
& = & \| P_{t_0} (|f|^2\omega_n) \| _{L^2(X,\mu_{\varphi })} \times \| \Delta_{\varphi}P_{t_0} (|f|^2\omega_n) \| _{L^2(X,\mu_{\varphi })} \\
& \leq & \|f\|^4_{\infty}\| \omega_n \| _{L^2(X,\mu_{\varphi })}.
\end{eqnarray*}
The last inequality holds since $\Delta_{\varphi}$ is self-adjoint on $L^2$ and hence admits bounded functional calculus. We obtain similar results fo the operator $\vec{\Delta}_{\varphi}$  and so
\begin{eqnarray*}
\int_{X} |dR_0(x,t)|  d\mu_{\varphi } (x)  &\leq & C(\varepsilon) \left( \|f\|^4_{\infty}\| \omega_n \| _{L^2(X,\mu_{\varphi })} + \|\vec{g}\|^4_{\infty}\| \omega_n^{-1} \| _{L^2(X,\mu_{\varphi })} \right)
\end{eqnarray*}
As a consequence, the right-hand side integral of (\ref{Tun3}) is finite. Letting $l$ tend to infinity implies (\ref%
{Tun2}) and concludes the proof of the proposition. 
\end{proof}
\subsection{Conclusion}

\begin{proof}[Proof of Theorem \ref{thm1}]
To finish the proof of the theorem, we use a standard trick. Indeed, by combining the reverse Fatou lemma as well as Propositions \ref{below} and \ref{above} and passing to the limit as $\varepsilon$ tends to $0$ we get
\[\int_{0}^{\infty} \int_{X}|\overline{\nabla} P_{t}f(x)||\overline{\nabla} \vec{P_{t}}\vec{g}(x)| t \, d\mu_{\varphi}(x) \deriv t \leq 20\widetilde Q_{2}(\omega_n)(\|f\|^2_{L^{2}(X,\omega_n\mu_{\varphi})}+ \|\vec{g}\|^2_{L^{2}(T^{*}X,\omega_n^{-1}\mu_{\varphi})}).\]
We now apply the above inequality to $\lambda f$ and $\lambda^{-1} \vec{g}$ instead of $f$ and $\vec{g}$ and then minimize the result in $\lambda>0$.
\end{proof}

\section{Enlarging the set of weights}
Now that we have boundedness results for a certain type of weights in $L^2(X,\mu_{\varphi})$, we will enlarge the set of weights $\omega$ satisfying Theorem \ref{thm1} to all $L^1_{loc}(X,\mu_{\varphi})$, provided that $\widetilde{Q}_2(\omega)$ is well defined and finite. We heavily use in the following proof the fact that constants are in $L^2(X,\mu_{\varphi})$\footnote{This condition implies that $\mu_{\varphi}(X)$ is finite and therefore the kernel of $ -\Delta_{\varphi}$ is the set of constant functions on $X$.}. The main problem would be defining $P_t \omega$ when $\omega \in L^1_{loc}(X,\mu_{\varphi})$. We will proceed as follows:\\
As stated in Remark \ref{remark1}, take $\omega \in L^1_{loc}(X,\mu_{\varphi})$ and define its two-sided truncation 
\[\omega _n = n^{-1} \chi _{\omega \leq n^{-1}} + \omega \chi _{n^{-1} \leq \omega \leq n} + n \chi _{\omega \geq n}, \]
where $\chi$ is the characteristic function. Then we have the following properties
\begin{equation}
\omega _n \in L^2(X,\mu_{\varphi}); \label{L2}
\end{equation}
\begin{equation}
\widetilde{Q}_2(\omega _n) \leq \widetilde{Q}_2(\omega); \label{Q2}
\end{equation}
\begin{equation}
\widetilde{Q}_2(\omega _n) \underset{n\to +\infty}{\longrightarrow}  \widetilde{Q}_2(\omega). \label{lim_omega_n}
\end{equation}
This means that we can approximate a function in the class $\widetilde{A}_2$ by bounded functions from the same class, with control of their $\widetilde{A}_2$ constants.\\
Property (\ref{L2}) is immediate because constant functions are 
integrable with respect to our measure. Besides, $\omega _n \underset{n\to +\infty}{\longrightarrow} \omega$ and consequently, the definition of $P_t \omega$ arises naturally by posing $P_t \omega = \lim \limits_{n \to \infty}P_t \omega _n$. This limit exists and makes sense because $P_t \omega \leq \varliminf_{n \to \infty} P_t \omega _n$ by Fatou's lemma. Furthermore, 
\[P_t \omega _n \leq \dfrac{1}{n} + P_t(\omega \chi _{n^{-1} \leq \omega \leq n} + n \chi _{\omega \geq n}).\]
The first term tends to zero as $n$ tends to infinity and the term between brackets is increasing in $n$, which means that we can use the monotone convergence theorem to obtain
\begin{eqnarray*}
P_t \omega &\leq & \varliminf_{n \to \infty} P_t \omega _n \\
& \leq & \varlimsup \limits_{n \to \infty}P_t \omega _n \\
& \leq & \varlimsup \limits_{n \to \infty} \left( \dfrac{1}{n} + P_t(\omega \chi _{n^{-1} \leq \omega \leq n} + n \chi _{\omega \geq n}) \right) \\
& =& P_t \omega .
\end{eqnarray*} 
We need the following preliminary lemma to prove properties (\ref{Q2}) and (\ref{lim_omega_n}) :

\begin{lemma}
Let $\omega ^n = \omega \chi _{\omega \leq n} + n \chi _{\omega \geq n}$. Then we have $\widetilde{Q}_2(\omega ^n) \leq \widetilde{Q}_2(\omega)$ and $\widetilde{Q}_2(\omega ^n) \underset{n\to +\infty}{\longrightarrow}  \widetilde{Q}_2(\omega)$.
\end{lemma}

\begin{proof}
If $\omega ^n = \omega \chi _{\omega \leq n} + n \chi _{\omega \geq n}$, then $(\omega ^n)^{-1} = \omega ^{-1} \chi _{\omega ^{-1} \geq n^{-1}} + n^{-1} \chi _{\omega ^{-1} \leq n^{-1}}$. \\
Thus, we need to prove that
\[P_t \omega P_t \omega ^{-1} - P_t \omega ^n P_t (\omega ^n)^{-1} \geq 0.\]
Write $P_t \omega = P_t \left( \omega \chi_{\omega \leq n} \right)  + P_t \left( \omega \chi_{\omega > n} \right)$ and denote 
\begin{eqnarray*}
A&=&P_t \left( \omega \chi_{\omega \leq n}\right); A^{-1}=P_t \left( \omega ^{-1} \chi_{\omega \leq n}\right);\\
B&=&P_t \left( \omega \chi_{\omega > n}\right);  B^{-1}=P_t \left(\omega ^{-1}\chi_{\omega > n}\right);\\
B_n&=&nP_t\left( \chi_{\omega > n}\right); B_n^{-1}=n^{-1}P_t\left( \chi_{\omega > n} \right).
\end{eqnarray*}
so that
\begin{eqnarray*}
P_t \omega P_t \omega ^{-1} - P_t \omega ^n P_t (\omega ^n)^{-1} &=& (A+B)(A^{-1}+B^{-1})-(A+B_n)(A^{-1}+B_n^{-1}) \\
&=& A(B^{-1}-B_n^{-1})+A^{-1}(B-B_n)\\
&& +(BB^{-1}-B_nB_n^{-1}).
\end{eqnarray*}
The last term between brackets is positif thanks to H\"{o}lder's and Jensen's inequalities. For the other terms, notice that 
\begin{eqnarray*}
B^{-1}-B_n^{-1} &=& \int p_t(x,y) \omega ^{-1}(y)\chi _{\omega >n}(y) d\mu_{\varphi}(y) -\int n^{-1}p_t(x,y) \chi _{\omega >n}(y) d\mu_{\varphi}(y)   \\
&=&  \int p_t(x,y)(n\omega ^{-1}(y) n^{-1} - \omega (y) \omega ^{-1} (y)n^{-1})\chi _{\omega >n}(y) d\mu_{\varphi}(y) \\
&=& \int p_t(x,y)(n - \omega (y)) n^{-1}\omega ^{-1} (y)\chi _{\omega >n}(y) d\mu_{\varphi}(y) 
\end{eqnarray*}
and analogously,
\[B-B_n= \int p_t(x,y)(\omega (y)-n) \chi _{\omega >n}(y) d\mu_{\varphi}(y)  \]
Hence,
\[P_t \omega P_t \omega ^{-1} - P_t \omega ^n P_t (\omega ^n)^{-1} \geq \int p_t(x,y)\left( \dfrac{\omega (y) -n}{n\omega (y)} \left( n \omega (y) A^{-1}- A \right) \right)\chi _{\omega >n}(y)d\mu_{\varphi}(y)\]
The kernel $p_t$ being positif, the integral on the right side is positif too, since $\omega > n$, $A \leq n$ and $A^{-1} \geq n^{-1}$.\\
The second part of the lemma follows from Fatou's lemma, \cite[Section 4]{rez}. Indeed, $(\omega ^n)_n$ is a positif sequence that converges to $\omega$ and we have
\begin{eqnarray*}
P_t \omega P_t \omega ^{-1} &= & \int p_t(x,y)\omega (y)d\mu_{\varphi}(y)\int p_t(x,y)\omega ^{-1} (y)d\mu_{\varphi}(y) \\
&\leq & \varliminf_{n \to \infty}  \int p_t(x,y)\omega ^n (y)d\mu_{\varphi}(y) \int p_t(x,y)(\omega ^n) ^{-1} (y)d\mu_{\varphi}(y) \\
& \leq & \varliminf_{n \to \infty}  \widetilde{Q}_2(\omega ^n) \\
& \leq & \varlimsup_{n \to \infty}  \widetilde{Q}_2(\omega ^n) \\
& \leq & \widetilde{Q}_2(\omega ), \text{ by first part of the lemma.}
\end{eqnarray*}
Taking supremum over $(x,t)\in X \times \mathbb{R}_+$ on the left-hand side finishes the proof.
\end{proof}
Now, to prove (\ref{Q2}) and (\ref{lim_omega_n}), apply the previous lemma to $(\omega^n) ^{-1}$.\\

As a consequence, if we let $n$ tend to infinity in (\ref{E1}), it only remains to prove that 
\[\|f\|_{L^2(X,\omega _n\mu_{\varphi})} \underset{n\to +\infty}\longrightarrow \|f\|_{L^2(X,\omega\mu_{\varphi})} \text{ and } \|\vec{g}\|_{L^2(T^*_xX,\omega^{-1} _n\mu_{\varphi})} \underset{n\to +\infty}\longrightarrow \|\vec{g}\|_{L^2(T^*_xX,\omega^{-1}\mu_{\varphi})}.\]
To do that, we are going to use the dominated convergence theorem. Indeed, by construction of $\omega _n$ and $\omega^{-1} _n$, we know that
\[\omega _n \leq \omega +1 \text{ and } \omega^{-1} _n \leq \omega^{-1} + 1.\]
We can pass to the limit in $n$ since $f\in L^2(X,\mu_{\varphi}) \cap L^2(X,\omega\mu_{\varphi})$ and $\vec{g} \in L^2(T^*_xX,\omega^{-1}\mu_{\varphi}) \cap L^2(T^*_xX,\mu_{\varphi}) $.\\
We can also recover Corollary \ref{coro} by using Formula (\ref{3}) and pass to the limit in $n$ in (\ref{h(t)}), again by the dominated convergence theorem.

\begin{remark}[Remark on the sharpness of the result]
To prove the sharpness of the result, notice that in the particular case when $\varphi=0$, the Hilbert transform on the unit disk is an operator that falls into our framework and the result is already known to be sharp \cite{PW}, meaning that the estimate in terms of dependence on weights is linear and can't be improved. 
\end{remark}

\section{Case of the Gauss space}
We now present a concrete example for the previous weighted estimates, namely the Gauss space, which is obtained when $X=\mathbb{R}^n$ and $\varphi (x) = \dfrac{|x|^2}{2}$. We then have the gaussian measure 
\[d\gamma (x) = (2\pi)^{-n/2} \exp (-\frac{|x|^2}{2})dx\] 
on $\mathbb{R}^n$ and the Ornstein-Uhlenbeck operator \footnote{The Ornstein-Uhlenbeck operator is an integral operator that admits a kernel called the Mehler kernel, which can be given by an explicit representation.}
\[Lf(x) = \Delta f(x) - x. \nabla f(x) \]
on $L^2(\mathbb{R}^n, d\gamma)$. This operator generates a diffusion semigroup $P_t$, which has been the object of many investigations during the last decades. \\
Note also that $\Ric_{L} \geq 0.$\\

If we define by $\mathcal{R}_L = d \circ (-L)^{-1/2}$ the Riesz transform associated to the Ornstein-Uhlenbeck semigroup, then

\begin{corollary}
For all $f\in \overline{L^{2}(\mathbb{R}^n,\omega\gamma)\cap R(-L)}^{L^2}$ and $\omega,\omega^{-1} \in L^1_{loc}(\mathbb{R}^n,\gamma)$ such that $\omega > 0$ $\gamma$-a.e we have
\[ \|\mathcal{R}_{L}f\|_{L^{2}(\mathbb{R}^n,\omega\gamma)}\leq 80\widetilde
Q_{2}(\omega)\|f\|_{L^{2}(\mathbb{R}^n,\omega\gamma)}. \]
\end{corollary}

\end{document}